\documentclass[psamsfonts]{amsart}
\usepackage[utf8]{inputenc}
\usepackage{prettyref}
\usepackage{mathtools}
\usepackage{amstext}
\usepackage{amsthm}
\usepackage{amssymb}

\makeatletter
\numberwithin{equation}{section}
\numberwithin{figure}{section}

\usepackage{amsfonts}\usepackage[all,arc]{xy}
\usepackage{enumerate}
\usepackage{mathrsfs}

\newtheorem{thm}{Theorem}[section]
\newtheorem{cor}[thm]{Corollary}

\theoremstyle{definition}

\theoremstyle{remark}
\newtheorem{rem}[thm]{Remark}

\numberwithin{equation}{section}

\makeatother

\begin{document}
\title[Inverse spectral results]{Inverse spectral results for the magnetic Schrödinger operator in
low dimensions }
\author{Devon Farrell, Anna Holloway \& Nikhil Savale}
\address{School of Mathematics, Trinity College Dublin, Dublin 2, Ireland}
\email{farred22@tcd.ie, hollowan@tcd.ie \& savalen@tcd.ie}
\thanks{The work presented here was carried out while D.F. and A.H. were students
in the Hamilton Trust Summer Research Internship of the school of
mathematics at Trinity College Dublin; with N.S. as their mentor.
It is a pleasure to acknowledge the financial support of the Hamilton
Trust.}
\begin{abstract}
We consider the semiclassical magnetic Schrödinger operator in dimensions
one and two. The first result is that an even one dimensional semiclassical
potential is spectrally determined by the Schrödinger operator. The
second result states that a two dimensional radial magnetic field
is spectrally determined by a purely magnetic Schrödinger operator.
This extends the earlier work of Guillemin-Wang \cite{Guillemin-Wang-2012}
and answers a conjecture of theirs therein.
\end{abstract}

\maketitle
\tableofcontents{}

\section{Introduction}

For the Laplacian on a bounded domain $\Omega\subset\mathbb{R}^{n}$,
one studies the eigenvalue problem with Dirichlet boundary conditions.
The boundedness of $\Omega$ confines the eigenfunctions $u$, ensuring
that the spectrum of $\Delta$ is discrete and grows to infinity.
In fact, the geometry of $\Omega$ is deeply encoded in this spectrum,
prompting Kac's famous "Can you hear the shape of a drum?" question
\cite{Kac1966} : inferring information about the shape of a drum
from the frequencies it can vibrate at, or more generally, inferring
information about $\Omega$ from the spectrum of the Laplacian in
the space. 

\subsection{Schrödinger operator}

Analogously, one can study the eigenvalue problem 
\begin{equation}
S_{h}u=\lambda u,
\end{equation}
where $S_{h}$ is the semiclassical Schrödinger operator given by
\begin{equation}
S_{h}=-h^{2}\Delta+V(x),\label{eq:Schrodinger operator}
\end{equation}
where $V(x)\in C^{\infty}(\mathbb{R}^{n})$ is a potential function
and $\Delta$ is the Laplacian on $\mathbb{R}^{n}$. We say $V$ is
a potential well if $|V(x)|\to+\infty$ as $|x|\to+\infty$. A result
of K. Friedrichs gives us that a potential well $V(x)$ confines eigenfunctions
$u$, ensuring that the spectrum \{$\lambda_{j}$\} of $S_{h}$ is
discrete and accumulates at infinity. One can similarly infer information
about $S_{h}$ from this spectrum, namely information about the potential
$V$. This is an analogous problem to Kac's question for Schrödinger
operators. It is furthermore of fundamental interest in quantum physics. 

In \cite{Guillemin-Uribe-2007}, Guillemin and Uribe show that any
real analytic potential $V$, symmetric with respect to all coordinate
axes and with a unique global minimum, is determined within the class
of all such potentials by the spectrum at the bottom of the well of
its semiclassical Schrödinger operator. The discreteness of the spectrum
\{$\lambda_{j}$\} allows the operator trace to be expressed as a
countable sum. In \cite{Guillemin-Wang-2012} (see also \cite{Dimassi-Sjostrand}),
such an asymptotic trace expansion 
\begin{equation}
\textnormal{tr}(f(S_{h}))\sim(2\pi h)^{-n}\sum^{\infty}_{k=0}\nu_{k}(f)h^{2k},\label{eq:functional trace expansion Schrodinger}
\end{equation}
for $f\in C^{\infty}_{c}\left(\mathbb{R}\right)$, is given, where
the coefficients $\nu_{k}(f)$ are semiclassical spectral invariants,
determined entirely by the spectrum of $S_{h}$. A recursive algorithm
for computing higher order coefficients is also given.

The problem of recovering the potential $V$ from the Schrödinger
operator is fundamental in mathematical physics. It is already non-trivial
in one dimension. In \cite{cdv-guillemin-2011} a form of the expansion
is used to obtain an inverse spectral result for $V$: showing that
if $c\in[0,a)$ is an isolated critical value of $V$ and $V^{-1}(c)$
is a single non-degenerate critical point, $p$, then the first two
terms of the expansion determine the Taylor series of $V$ at $p$,
and hence, if $V$ is analytic in a neighbourhood of $p$, determine
$V$ itself in this neighbourhood. In \cite[Thm. 1.2]{Guillemin-Wang-2012},
Guillemin and Wang compute terms of this expansion for the perturbed
Schrödinger operator 
\begin{equation}
P_{h}=-h^{2}\Delta+V(x)+h^{2}V_{2}(x)
\end{equation}
obtaining inverse spectral results for $V$ and $V_{2}$. Namely,
they show that both $V$ and $V_{2}$ are spectrally determined under
certain asymmetry conditions on them.

Our first main result uses the same trace expansion in \prettyref{eq:functional trace expansion Schrodinger}
to extend the above theorem from perturbations of the form $V(x)+h^{2}V_{2}(x)$
to arbitrary semiclassical potentials of the form 
\begin{equation}
V_{h}\sim V_{0}+hV_{1}+h^{2}V_{2}+\dots.\label{eq:semiclassical expansion}
\end{equation}
 Our theorem is the following.
\begin{thm}
\label{thm: Theorem 1} Suppose the one dimensional potential function
$V_{h}$ has a semiclassical expansion \prettyref{eq:semiclassical expansion},
where $V_{0}$ is a potential well with each $V_{k}$ is even and
$V_{0}'\left(x\right)>0$ for $x>0$. Then $V_{0},V_{1},V_{2},\dots$are
determined by the spectrum of the Schrödinger operator $S_{h}:=-h^{2}\Delta+V_{h}(x)$.
\end{thm}

This had earlier been conjectured by Guillemin-Wang \cite[pg. 105]{Guillemin-Wang-2012}.
As a Corollary, one obtains a generalization of the result of \cite{cdv-guillemin-2011}
for $h-$dependent potentials.
\begin{cor}
Suppose $V_{h}$ is an even, $h-$dependent one-dimensional potential
that is real analytic in $h\in\left[0,1\right]$. Suppose $V_{0}$
is a potential well with $V_{0}'\left(x\right)>0$ for $x>0$. Then
$V_{h}$ is determined by the spectrum of the Schrödinger operator
\prettyref{eq:Schrodinger operator} $S_{h}:=-h^{2}\Delta+V_{h}(x)$.
\end{cor}

It is likely that our methods work for radial semiclassical potentials
$V_{h}$ in arbitrary $n$- dimensions. 

\subsection{Magnetic Schrödinger operator}

Analogous inverse spectral questions arise for Schrödinger operators
in a magnetic field, 
\begin{equation}
\hat{S}_{h}=\sum^{n}_{j=1}\left(ih\frac{\partial}{\partial x_{j}}+a_{j}\right)^{2}+V(x)\label{eq:magnetic schrodinger}
\end{equation}
as in \cite[Sec. 7]{Guillemin-Wang-2012}, where 
\begin{equation}
B=da=\sum_{j<l}(\frac{\partial}{\partial x_{l}}a_{j}-\frac{\partial}{\partial x_{j}}a_{l})\,dx_{j}\wedge dx_{l}
\end{equation}
is the associated magnetic field. First, we consider the purely magnetic
case: when $V(x)=0$. We say $B$ is a magnetic potential well if
$|B(x)|\to+\infty$ as $|x|\to+\infty$. For such magnetic fields
the operator $\hat{S}_{h}$ is known to have a discrete spectrum \cite[Ch. 3]{Agmon-book-1982}.

Furthermore, it follows from standard methods (eg. \cite[Thm. 8]{Marinescu-Savale18}
, \cite[Thm. 1.1]{Ma-Marinescu-Zelditch-2015}) that, for such $\hat{S}_{h}$,
one has the trace expansion 
\begin{equation}
\textnormal{Tr}f(e^{-\frac{t\hat{S}_{h}}{h}})\sim h^{-n/2}[c^{f}_{0}+c^{f}_{1}h+c^{f}_{2}h^{2}+\dots],\label{eq:rescaled heat kernel exp.}
\end{equation}
for $f\in C^{\infty}_{c}\left(\mathbb{R}\right)$, where the coefficients
$c_{k}\left(f\right)$ are semiclassical spectral invariants. The
leading term can be calculated. In the case when $B=|B(x)|$ is radial,
it is seen to be 
\begin{equation}
c_{0}=\int^{\infty}_{0}r\tilde{f}(B)\,dr\label{eq:resc. heat kernel 1t 2D}
\end{equation}
in the $2$-dimensions. And 
\begin{equation}
c_{0}=\int^{\infty}_{0}r^{2}\tilde{f}(B)\,dr\label{eq:resc. heat kernel 1t 3D}
\end{equation}
in the 3-dimensions. Here 
\begin{equation}
\tilde{f}(B)=\begin{cases}
B[f(B)+f(3B)+f(5B)+\dots] & \text{if }B>0\\
\int^{\infty}_{0}f(s)\,ds & \text{if }B=0.
\end{cases}.\label{eq:1t component}
\end{equation}

We then use these coefficients to obtain the following inverse spectral
result in 2 dimensions:
\begin{thm}
\label{thm: Theorem 3} Let $B\in\Omega^{2}$ be the associated magnetic
field of the purely magnetic Schrödinger operator $\hat{S}_{h}$ \prettyref{eq:magnetic schrodinger}
in 2 dimensions. Suppose $B\left(r\right)=\left|B\right|$ is radially
symmetric magnetic field, with $B^{\prime}(r)>0$ and $B(r)\to\infty$
as $r\to\infty$. Then the spectrum of $\hat{S}_{h}$ uniquely determines
$B(r)$.
\end{thm}

In the above result, the potential $V=0$ is zero. It had been established
in \cite[Thm. 1.3]{Guillemin-Wang-2012} in the presence of a radial
confining potential $V$. Hence our result extends their work therein.

The article is structured as follows. In \prettyref{sec:The-Semiclassical-Trace}
we begin with background on the semiclassical trace formulas used
in the proofs of \prettyref{thm: Theorem 1} and \prettyref{thm: Theorem 3}.
In \prettyref{sec:Spectral-Determination-of semiclassical potential}
we give a proof of \prettyref{thm: Theorem 1} based on the semiclassical
trace expansion for the Schrödinger operator. In \prettyref{sec:Spectral-Determination-of radial magnetic field}
we give a proof of \prettyref{thm: Theorem 3} using the rescaled
trace formula for the magnetic Laplacian. In the final section \prettyref{sec:Spectral-Determination-of via heat kernel}
we give an alternate proof of \prettyref{thm: Theorem 3} via the
heat trace, under some additional assumptions.

\section{The Semiclassical Trace Formulas}\label{sec:The-Semiclassical-Trace}

In this section we give an outline of the derivation of the trace
formulas used in our proofs.

\subsection{Semiclassical trace expansion }

First we consider the trace formula for the Schrödinger operator \prettyref{eq:Schrodinger operator}.
In fact what follows works for any elliptic semiclassical pseudodifferential
operator. Our exposition closely follows the one of Guillemin and
Wang.

A semiclassical differential operator of order $r$ on $\mathbb{R}^{n}$
is an operator on $C^{\infty}(\mathbb{R}^{n})$ of the form 
\begin{equation}
A=\sum_{|\alpha|\leq r}a_{\alpha}(x,h)(hD_{x})^{\alpha},
\end{equation}
where $a_{\alpha}(x,h)$ are smooth functions on $\mathbb{R}^{n}\times\mathbb{R}$
and $hD^{\alpha}_{x}=(\frac{h}{i})^{|\alpha|}\partial^{\alpha}_{x}$,
for $\alpha=(\alpha_{1},...,\alpha_{n})$ a multi-index, with $|\alpha|=\alpha_{1}+\dots+\alpha_{n}.$
We let $P_{h}$ be such an operator on $\mathbb{R}^{n}$. We consider
two phase-space representations of $P_{h}$, namely its Kohn-Nirenberg
symbol 
\begin{equation}
p(x,\xi,h)=\sum_{\alpha}a_{\alpha}(x,h)\xi^{\alpha}
\end{equation}
and Weyl symbol 
\begin{equation}
p^{w}(x,\xi,h)=\exp(-\frac{h}{2}D_{\xi}\partial_{x})p(x,\xi,h).
\end{equation}
We assume that $p^{w}$ is real valued, so that $P_{h}$ is self-adjoint
\cite{Zworski}. We also assume that $(p^{w})^{-1}([a,b])$ is compact
for some interval $[a,b]\subset\mathbb{R}$, $0\leq h\leq h_{0}$.
We then have, by Friedrich's theorem, that the spectrum of $P_{h}$,
for $h<h_{0}$, on $[a,b]$, consists of a finite number of eigenvalues,
\begin{equation}
\lambda_{i}(h),\;\;\;1\leq i\leq N(h),
\end{equation}
where $N(h)\to\infty$ as $h\to0$. Let 
\begin{equation}
p(x,\xi,0)=p^{w}(x,\xi,0)
\end{equation}
be the principal symbol of $P_{h}$, with asymptotic expansion $p=\sum p_{j}h^{j}$.
Suppose $f\in C^{\infty}_{0}(\mathbb{R})$ is smooth and compactly
supported on $(a,b)$. Then 
\begin{equation}
f(P_{h})=\frac{1}{\sqrt{2\pi}}\int\hat{f}(t)e^{itP_{h}}\,dt,
\end{equation}
where $\hat{f}$ is the Fourier transform of $f$. Hence, understanding
$f(P_{h})$ reduces to understanding the operator $U(t)=e^{itP_{h}}$.
Formally, we have that $U(t)$ satisfies 
\begin{equation}
\frac{1}{i}\frac{\partial}{\partial t}U(t)=P_{h}U(t),\quad U(0)=\operatorname{Id}.
\end{equation}
We look for $U(t)$ as a pseudodifferential operator. Let $K_{U(t)}(x,y,t,h)$
be the Schwartz kernel of $U(t)$. From \cite{Zworski}, for this
to be the kernel of a pseudodifferential operator, it must have the
form 
\begin{equation}
K_{U(t)}(x,y,t,h)=(2\pi h)^{-n}\int a(x,\xi,t,h)e^{i\frac{(x-y)\xi}{h}}\,d\xi,
\end{equation}
with initial condition $a(x,\xi,0,h)=1$. We then use the ansatz 
\begin{equation}
a(x,\xi,t,h)=e^{itp_{0}(x,\xi)}b(x,\xi,t,h).
\end{equation}
The initial condition becomes $b(x,\xi,0,h)=1$. Substituting this
ansatz into (2.7) gives us 
\begin{equation}
\frac{1}{i}\frac{\partial}{\partial t}\left(e^{itp_{0}(x,\xi)}b(x,\xi,t,h)\right)=p(x,\xi,h)\,\star\,\left(e^{itp_{0}(x,\xi)}b(x,\xi,t,h)\right),
\end{equation}
where $\star$ denotes the composition of pseudodifferential operators.
Expanding this out gives us 
\begin{equation}
e^{itp_{0}}\left(\frac{1}{i}\frac{\partial b}{\partial t}+p_{0}b\right)\sim\sum_{\alpha}\frac{h^{\alpha}}{\alpha!}D^{\alpha}_{\xi}p\partial^{\alpha}_{x}(e^{itp_{0}}b),
\end{equation}
and cancelling the $e^{itp_{0}}$ factor on both sides, we get 
\begin{equation}
\frac{1}{i}\frac{\partial b}{\partial t}\sim\sum_{|\alpha|\geq1}\frac{h^{\alpha}}{\alpha!}D^{\alpha}_{\xi}pQ^{\alpha}b+(p-p_{0})b,
\end{equation}
where $Q^{\alpha}$ is the operator 
\begin{equation}
Q^{\alpha}=\frac{1}{\alpha!}\left(\partial_{x}+it\frac{\partial p_{0}}{\partial x}\right)^{\alpha}.
\end{equation}
Now, expanding $b$ and $p$ in powers of $h$, 
\begin{equation}
b\sim\sum_{k}b_{k}(x,\xi,t)h^{k},\quad p\sim\sum_{k}p_{k}h^{k},
\end{equation}
and matching powers of $h$ gives us 
\begin{equation}
\frac{1}{i}\frac{\partial b_{m}}{\partial t}=\sum_{|\alpha|\geq1}\sum_{j+k+|\alpha|=m}D^{\alpha}_{\xi}p_{j}Q^{\alpha}b_{k}+\sum_{j\geq1}p_{j}b_{m-j},
\end{equation}
with $b_{0}(x,\xi,t)=1$ and $b_{m}(x,\xi,0)=0$ for $m\geq1$. Each
$b_{m}$ is then computed recursively using integration. It can be
shown inductively that each $b_{m}$ is a polynomial in $t$ of degree
$2m$, 
\begin{equation}
b_{m}(x,\xi,t)=\sum_{l\leq2m}b_{m,l}(x,\xi)t^{l}.
\end{equation}

From \cite{Zworski}, we have that the trace of a pseudodifferential
operator $A$ on $\mathbb{R}^{n}$ with Kohn-Nirenberg symbol $a(x,\xi,h)$
is obtained by integrating the Schwartz kernel $K_{A}(x,y)$ along
the diagonal $x=y$: 
\begin{equation}
\operatorname{Tr}(A)=\int K_{A}(x,x)\,dx
\end{equation}
so it follows from the Fourier representation of $f(P_{h})$ that,
up to constants 
\begin{align*}
\operatorname{Tr}(f(P_{h})) & =\int\hat{f}(t)\,K_{U(t)}(x,x,t,h)\,dx\,dt\\
 & =h^{-n}\int\hat{f}(t)e^{itp_{0}}b(x,\xi,t,h)\,dx\,d\xi\,dt\,+O(h^{\infty})\\
 & =h^{-n}\int\hat{f}(t)e^{itp_{0}}\sum_{k}b_{k}(x,\xi,t)h^{k}\,dx\,d\xi\,dt\,+O(h^{\infty})\\
 & =h^{-n}\int\hat{f}(t)e^{itp_{0}}\sum_{k}\sum_{l\leq2k}b_{k,l}(x,\xi)t^{l}h^{k}\,dx\,d\xi\,dt\,+O(h^{\infty}).
\end{align*}
Recall that we have the Fourier transform identity 
\begin{equation}
\widehat{g^{\prime}}(\xi)=i\xi\hat{g}(\xi),
\end{equation}
so each $t^{l}e^{itp_{0}}$, when integrated against $\hat{f}(t)$,
gives 
\begin{equation}
\int\hat{f}(t)t^{l}e^{itp_{0}}\,dt=\left(\frac{1}{i}\frac{d}{ds}\right)^{l}f(s)|_{s=p_{0}}.
\end{equation}
Finally, from this, we get 
\begin{align*}
\operatorname{Tr}(f(P_{h})) & =h^{-n}\int\hat{f}(t)e^{itp_{0}}\sum_{k}\sum_{l\leq2k}b_{k,l}(x,\xi)t^{l}h^{k}\,dx\,d\xi\,dt\,+O(h^{\infty})\\
 & =h^{-n}\sum_{k}h^{k}\int\sum_{l\leq2k}b_{k,l}(x,\xi)\left(\frac{1}{i}\frac{d}{ds}\right)^{l}f(s)|_{s=p_{0}}\,dx\,d\xi\\
 & =h^{-n}\sum_{k}\nu_{k}(f)h^{k},
\end{align*}
where the $k^{th}$ spectral invariant, $\nu_{k}(f)$, is given by
\begin{equation}
\nu_{k}(f)=\int\sum_{l\leq2k}b_{k,l}(x,\xi)\left(\frac{1}{i}\frac{d}{ds}\right)^{l}f(s)|_{s=p_{0}}\,dx\,d\xi.
\end{equation}

\subsection{Rescaled heat trace expansion}

We now recall the rescaled heat trace expansion. The results of this
section are best suited for the purely magnetic Schrödinger given
by
\begin{equation}
\hat{S}_{h}=\sum^{n}_{j=1}\left(ih\frac{\partial}{\partial x_{j}}+a_{j}\right)^{2}.
\end{equation}
The magnetic field $B$ is the differential of the magnetic potential
\begin{equation}
a=\sum^{n}_{j=1}a_{j}\,dx_{j};\quad\ B=da=\sum_{j<l}(\frac{\partial}{\partial x_{l}}a_{j}-\frac{\partial}{\partial x_{j}}a_{l})\,dx_{j}\wedge dx_{l}.
\end{equation}
It has the usual norm $|B|=\sqrt{\sum_{i,j}|B_{i,j}|^{2}}$.

We begin by recalling from \cite[Thm. 8]{Marinescu-Savale18} , \cite[Thm. 1.1]{Ma-Marinescu-Zelditch-2015}
the heat trace expansion for a purely magnetic Schrödinger operator:
if $|B|\to\infty$ as $|x|\to\infty$ then $\hat{S}_{h}$ admits the
heat kernel expansion 
\begin{equation}
e^{-\frac{t\hat{S}_{h}}{h}}\left(x,x\right)\sim h^{-n/2}[c_{0}(x;t)+c_{1}(x;t)h+c_{2}(x;t)h^{2}+\dots].
\end{equation}

Since our later arguments rely only on the leading coefficient $c_{0}(x;t)$,
we include a brief sketch of its derivation: \\
 In the 2-dimensional case, one has that $e^{\frac{-t\hat{S}_{h}}{h}}u_{0}$
solves 
\begin{equation}
(h\partial_{t}+\hat{S}_{h})u=0;\quad u(x,0)=u_{0},
\end{equation}
and has Schwartz kernel $K^{h}_{t}(x,y)$, satisfying the equations
\begin{equation}
(h\partial_{t}+\hat{S}_{h})K^{h}_{t}(x,y)=0
\end{equation}
and 
\begin{equation}
\lim_{t\rightarrow0}K^{h}_{t}(x,y)=\delta(x-y).
\end{equation}
The trace is then obtained by integrating this kernel over the diagonal:
\begin{equation}
\operatorname{Tr}(e^{\frac{-t\hat{S}_{h}}{h}})=\int K^{h}_{t}(x,x)\,dx.
\end{equation}
To understand $K^{h}_{t}(x,x)$ near a point $x_{0}$, we translate
to the origin so that $x_{0}=0$. Now, locally choosing a an appropriate
gauge, we can write 
\begin{equation}
a(x)=\frac{1}{2}B(0)(-x_{2},x_{1})+\mathcal{O}(|x|^{2}),
\end{equation}
and hence 
\begin{align*}
\hat{S}_{h} & =\sum^{n}_{j=1}\left(ih\frac{\partial}{\partial x_{j}}+a_{j}\right)^{2}\\
 & =(ih\nabla_{x}+\frac{1}{2}B(0)(-x_{2},x_{1})+\mathcal{O}(|x|^{2}))^{2}.
\end{align*}
We can then rescale $x\to\sqrt{h}x$, with $\partial_{x}\to\frac{1}{\sqrt{h}}\partial_{x}$,
obtaining 
\begin{align*}
\hat{S}_{h} & =(i\sqrt{h}\nabla_{x}+\frac{1}{2}B(0)(-\sqrt{h}x_{2},\sqrt{h}x_{1})+\mathcal{O}(h))^{2}\\
 & =h(i\nabla_{x}+\frac{1}{2}B(0)(-x_{2},x_{1}))^{2}+\mathcal{O}(h^{3/2}),
\end{align*}
and dividing by $h$ yields 
\begin{equation}
\frac{\hat{S}_{h}}{h}=(i\nabla_{x}+\frac{1}{2}B(0)(-x_{2},x_{1}))^{2}+\mathcal{O}(\sqrt{h}).
\end{equation}
Thus, to leading order near each point $x_{0}$, the heat kernel is
approximated by that of the constant magnetic field model $P_{0}:=(i\nabla_{x}+\frac{1}{2}B(0)(-x_{2},x_{1}))^{2},$
whose heat kernel on the diagonal is given by Mehler's formula as
\begin{equation}
\frac{1}{4\pi t}\frac{tB(0)}{\operatorname{sinh}(tB(0))}=\frac{B(0)}{4\pi\operatorname{sinh}(tB(0))}.
\end{equation}
Integrating local contributions over $x_{0}\in\mathbb{R}^{2}$ yields
\begin{equation}
c_{0}(x;t)=\frac{B(x)}{e^{tB(x)}-e^{-tB(x)}},
\end{equation}
in the $2$-dimensional case.

The factor of $h^{-1}$ in the expansion arises from the rescaling
$x\to x_{0}+\sqrt{h}x$, since preserving the initial condition $\lim\limits_{t\rightarrow0}K^{h}_{t}(x,y)=\delta(x-y),$
requires the rescaled kernel to satisfy $\delta(\sqrt{h}(x-y))=h^{-n/2}\delta(x-y)$
since $\delta(cz)=|c|^{-n}\delta(z)$ in $n$ dimensions.

\bigbreak

A similar proof exists for the $3$-dimensional case, yielding 
\begin{equation}
c_{0}(x;t)=\frac{1}{\sqrt{t}}\frac{B(x)}{e^{tB(x)}-e^{-tB(x)}}.
\end{equation}

This heat kernel expansion can in fact be generalised to the trace
of arbitrary smooth and compactly supported functions $f\in C^{\infty}_{c}\left(\mathbb{R}\right)$
of $\hat{S}_{h}$. Namely, we have 
\begin{equation}
f(\frac{\hat{S}_{h}}{h})\left(x,x\right)\sim h^{-n/2}[c^{f}_{0}\left(x\right)+c^{f}_{1}\left(x\right)h+c^{f}_{2}\left(x\right)h^{2}+\dots],
\end{equation}
where 
\begin{equation}
c^{f}_{0}\left(x\right)=\tilde{f}(B)=\begin{cases}
B[f(B)+f(3B)+f(5B)+\dots] & \text{if }B>0\\
\int^{\infty}_{0}f(s)\,ds & \text{if }B=0
\end{cases}
\end{equation}
with $B=B\left(x\right)$. The proof of the above is an elementary
calculation via Laplace transforms (cf. \cite[Prop. 7.4]{Savale2017-Koszul}). 

The above combined with the discreteness of the spectrum of $\hat{S}_{h}$
now gives the heat trace expansion formulas \prettyref{eq:rescaled heat kernel exp.},
\prettyref{eq:resc. heat kernel 1t 2D}, \prettyref{eq:resc. heat kernel 1t 3D}
and \prettyref{eq:1t component}.

\section{\medspace{}Spectral Determination of a Semiclassical Potential
in one dimension}\label{sec:Spectral-Determination-of semiclassical potential}

Let $P_{h}=S_{h}$ be a perturbed Schrödinger operator in one dimension,
of the form $S_{h}=-h^{2}\Delta+V_{h}(x)$ with $V_{h}\sim V_{0}+hV_{1}+h^{2}V_{2}+\dots.$
We wish to show that each $V_{j}$ of the expansion of $V_{h}$ is
spectrally determined. The key observation will be that the $m^{th}$
spectral invariant depends linearly on $V_{m}$, modulo terms involving
only $V_{0},\dots,V_{m-1}$.

From the previous section, we have the semiclassical trace expansion
\begin{equation}
\operatorname{Tr}(f(S_{h}))\sim h^{-1}\left(\sum^{\infty}_{m=0}\nu_{m}(f)h^{m}\right),
\end{equation}
where 
\begin{equation}
\nu_{m}(f)=\int e^{itp_{0}}\hat{f}(t)b_{m}(x,\xi,t)\,dt\,dx\,d\xi,
\end{equation}
and 
\begin{equation}
p_{0}=\xi^{2}+V_{0}
\end{equation}
is the constant term of the asymptotic expansion of the principal
symbol $p(x,\xi,h)$ of $S_{h}$. We also have from (2.15) that $b_{m}$
satisfies the recursion relation 
\begin{equation}
\frac{\partial b_{m}}{\partial t}=2\xi\left(\frac{\partial}{\partial x}+itV^{\prime}_{0}\right)b_{m-1}+\left(\frac{\partial}{\partial x}+itV^{\prime}_{0}\right)^{2}b_{m-2}+\sum^{m}_{j=1}V_{j}b_{m-j},
\end{equation}
with $b_{0}(x,\xi,t)=1$ and $b_{m}(x,\xi,t)=0$ for $m<0$.\\
 \hspace*{0.5cm} Integrating this expression gives $b_{m}$ in the
form 
\begin{equation}
b_{m}(V_{0},\dots,V_{m})=F(V_{0},\dots,V_{m-1})+tV_{m}
\end{equation}
for some function $F$, and so from (2.20) it follows by induction
that we have $\nu_{m}=\nu_{m}(V_{0},\dots,V_{m})$ as a function of
the first $m+1$ $V_{j}$'s. Hence we have 
\begin{equation}
\nu_{m}=G(V_{0},\dots,V_{m-1})+\int e^{itp_{0}}\hat{f}(t)tV_{m}\,dt\,dx\,d\xi
\end{equation}
for some function $G$.
\begin{proof}[Proof of \prettyref{thm: Theorem 1}]
 We now show inductively that the spectral invariants $\nu_{0},\nu_{1},\dots,\nu_{m}$
determine $V_{0},V_{1},\dots,V_{m}$ for $V_{h}$ as in the theorem
statement. For our base case we compute $\nu_{0}(f)$ using (2.20):
\begin{align*}
\nu_{0}(f) & =\int e^{itp_{0}}\hat{f}(t)\,dt\,dx\,d\xi\\
 & =\int\left(\int e^{itp_{0}}\hat{f}(t)\,dt\right)\,dx\,d\xi\\
 & =\int f(p_{0})\,dx\,d\xi.
\end{align*}
Take $f=\chi_{[0,\lambda]}$ the characteristic function of $[0,\lambda]$.
\begin{rem} For $f=\chi_{[0,\lambda]}$, we smoothly approximate
$f$ by $f_{\epsilon}\in C^{\infty}_{0}(\mathbb{R})$ with $f_{\epsilon}\rightarrow f$
pointwise, and pass to the limit via monotone convergence. We use
$f=\chi_{[0,\lambda]}$ with this understanding. \end{rem}Thus we
obtain 
\begin{align*}
\nu_{0}(f) & =\int_{\xi^{2}+V_{0}(x)\leq\lambda}dx\,d\xi\\
 & =\int_{V_{0}(x)\leq\lambda}\int_{-\sqrt{\lambda-V_{0}(x)}\leq\xi\leq\sqrt{\lambda-V_{0}(x)}}d\xi\,dx\\
 & =2\int_{V_{0}(x)\leq\lambda}\sqrt{\lambda-V_{0}(x)}\,dx.
\end{align*}
Let $x=g_{1}(s)$ be the inverse function of $s=V_{0}(x)$, $x\in(0,a)$
and $x=g_{2}(s)$ be the inverse function of $s=V_{0}(-x)$, $x>0$.
Then since $V_{0}$ is a single well potential, we get 
\begin{align*}
\nu_{0}(f) & =2\int^{\lambda}_{0}\sqrt{\lambda-s}g^{\prime}_{1}(s)\,ds+2\int^{\lambda}_{0}\sqrt{\lambda-s}g^{\prime}_{2}(s)\,ds\\
 & =2\int^{\lambda}_{0}\sqrt{\lambda-s}(g^{\prime}_{1}(s)+g^{\prime}_{2}(s))\,ds
\end{align*}
We recall the fractional integration operation of Abel, 
\begin{equation}
J^{a}h(\lambda)=\frac{1}{\Gamma(a)}\int^{\lambda}_{0}(\lambda-t)^{a-1}h(t)\,dt
\end{equation}
for $a>0$, which satisfies $J^{a}J^{b}=J^{a+b}$, and note that 
\begin{equation}
2\int^{\lambda}_{0}\sqrt{\lambda-s}(g^{\prime}_{1}(s)+g^{\prime}_{2}(s))\,ds=\sqrt{\pi}J^{\frac{3}{2}}(g^{\prime}_{1}(s)+g^{\prime}_{2}(s)).
\end{equation}
Hence by applying $J^{\frac{1}{2}}$ and differentiating by $\lambda$
twice, one recovers $g^{\prime}_{1}(s)+g^{\prime}_{2}(s)$. Using
the evenness of $V_{0}$, we have $g_{1}(s)=g_{2}(s)\eqqcolon g(s)$
and hence $g^{\prime}_{1}(s)=g^{\prime}_{2}(s)=g^{\prime}(s)$, and
so integrating with $g(0)=0$ recovers $g$ and hence $V_{0}$.\\
 \hspace*{0.5cm}Now assume that $\nu_{0},\dots,\nu_{k}$ determine
$V_{0},\dots,V_{k}$ for $0\leq k\leq m-1$. Under this hypothesis,
it follows from (3.6) that 
\[
\int e^{itp_{0}}\hat{f}(t)tV_{m}\,dt\,dx\,d\xi=\nu_{m}-G(V_{0},\dots,V_{m-1})
\]
is determined by $\nu_{0},\dots,\nu_{m-1},\nu_{m}$, so to show that
$V_{m}$ is determined by $\nu_{0},\dots,\nu_{m-1},\nu_{m}$ it suffices
to recover $V_{m}$ from this expression. Taking $f^{\prime}=\chi_{[0,\lambda]}$
we get 
\begin{align*}
\int e^{itp_{0}}\hat{f}(t)tV_{m}\,dt\,dx\,d\xi & =\int f^{\prime}(p_{0})V_{m}\,dx\,d\xi\\
 & =\int_{\xi^{2}+V_{0}(x)\leq\lambda}V_{m}\,dx\,d\xi,
\end{align*}
and applying a similar argument as before we have 
\begin{align*}
\int_{\xi^{2}+V_{0}(x)\leq\lambda}V_{m}\,dx\,d\xi & =\int_{V_{0}(x)\leq\lambda}\int_{-\sqrt{\lambda-V_{0}(x)}\leq\xi\leq\sqrt{\lambda-V_{0}(x)}}V_{m}(x)\,d\xi\,dx\\
 & =2\int_{V_{0}(x)\leq\lambda}\sqrt{\lambda-V_{0}(x)}V_{m}(x)\,dx\\
 & =2\int^{\lambda}_{0}\sqrt{\lambda-s}V_{m}(g_{1}(s))g^{\prime}_{1}(s)\,ds+2\int^{\lambda}_{0}\sqrt{\lambda-s}V_{m}(-g_{2}(s))g^{\prime}_{2}(s)\,ds\\
 & =2\int^{\lambda}_{0}\sqrt{\lambda-s}(V_{m}(g_{1}(s))g^{\prime}_{1}(s)+V_{m}(-g_{2}(s))g^{\prime}_{2}(s))\,ds\\
 & =4\int^{\lambda}_{0}\sqrt{\lambda-s}V_{m}(g(s))g^{\prime}(s)\,ds,
\end{align*}
with $g_{1},g_{2},$ and $g$ as before. By similarly applying $J^{\frac{1}{2}}$
and differentiating by $\lambda$ twice, one recovers $V_{m}(g(s))g^{\prime}(s)$.
We have that $V_{0}$, and hence $g$ and $g^{\prime}$, is determined
by $\nu_{0}$, so from this we can recover $V_{m}(g(s))=V_{m}(x)$.
\end{proof}

\section{ Spectral Determination of a Radially Symmetric Magnetic Field in
2 Dimensions}\label{sec:Spectral-Determination-of radial magnetic field}

In this section we prove \prettyref{thm: Theorem 3}. The proof shall
use a rescaled trace expansion \prettyref{eq:rescaled heat kernel exp.}
of the Schrödinger operator $\hat{S}_{h}$ corresponding to an arbitrary
compactly supported function $f\in C^{\infty}_{c}\left(\mathbb{R}\right)$.

We assume that $B(r)$ is monotone increasing. This means that, in
the case of $B(0)>0$, we have $B>0$ and, in the case of $B(0)=0$,
we have $B=0$ at $r=0$ only. As a result, the $B=0$ case of $\tilde{f}$
can essentially be ignored.
\begin{proof}[Proof of \prettyref{thm: Theorem 3}]
 Take $f$ to be an arbitrary smooth, compactly supported, non-negative
function of $\hat{S}_{h}$. Let $r(B)$ be the inverse function of
$B(r)$. We can then rewrite the $c_{0}$ coefficient as 
\begin{align*}
c^{f}_{0} & =\frac{1}{2}\int^{\infty}_{0}[r^{2}]^{\prime}\tilde{f}(B)\,dB\\
 & =\frac{1}{2}\int^{\infty}_{0}[r^{2}]^{\prime}B[f(B)+f(3B)+f(5B)+\dots]\,dB\\
 & =\frac{1}{2}\int^{\infty}_{0}\sum^{\infty}_{n=0}[r^{2}]^{\prime}Bf((2n+1)B)\,dB\\
 & =\frac{1}{2}\sum^{\infty}_{n=0}\int^{\infty}_{0}B[r^{2}]^{\prime}f((2n+1)B)\,dB,
\end{align*}
where the final equality follows from Fubini's Theorem since all the
terms are non-negative. Now taking $f=\chi_{[0,\lambda]}$ (recalling
remark 3.1), we get 
\begin{equation}
c_{0}(\lambda)=\sum^{\infty}_{n=0}\frac{1}{2}\int^{\lambda/{(2n+1)}}_{0}B[r^{2}]^{\prime}\,dB.
\end{equation}
Let 
\begin{equation}
R(x)=\frac{1}{2}\int^{x}_{0}B[r^{2}]^{\prime}\,dB.
\end{equation}
Then we have 
\begin{equation}
c_{0}(\lambda)=\sum^{\infty}_{n=0}R\left(\frac{\lambda}{2n+1}\right).
\end{equation}
We note that this series converges since we have $|R(x)|\leq Cx^{2}$,
for $x\in[0,1]$. 

Furthermore, by differentiating with respect to $x$, integrating
with $r(0)=0$, and observing that $r(B)>0$, for all $B>0$, we see
that $R(x)$ uniquely determines $r(B)$. This reduces our problem
to that of showing $R(x)$ is uniquely determined by 
\begin{equation}
\tilde{R}(\lambda)=\sum^{\infty}_{n=0}R\left(\frac{\lambda}{2n+1}\right)=R(\lambda)+R(\lambda/3)+R(\lambda/5)+\dots,
\end{equation}
which we now prove: We have that 
\begin{align*}
\tilde{R}(\lambda) & =\sum^{\infty}_{n=0}R\left(\frac{\lambda}{2n+1}\right)\\
 & =\sum^{\infty}_{n=1}\beta(n)R\left(\frac{\lambda}{n}\right),
\end{align*}
where 
\begin{equation}
\beta(n)=\begin{cases}
0 & \text{if }n\text{ even}\\
1 & \text{if }n\text{ odd}
\end{cases}.
\end{equation}
We recall the generalized Möbius inversion formula \cite[pg. 40]{Apostol1976}:
for $\alpha$ a completely multiplicative function we have 
\begin{equation}
G(x)=\sum_{n\leq x}\alpha(n)F\left(\frac{x}{n}\right)\quad\text{if, and only if, }F(x)=\sum_{n\leq x}\alpha^{-1}(n)G\left(\frac{x}{n}\right),
\end{equation}
where $\alpha^{-1}$ is the Dirichlet inverse of $\alpha$. We calculate
the Dirichlet inverse of $\beta$ and show that, under our assumptions,
we can generalise the above formula to $R$ in terms of $\tilde{R}$.
Note that $\beta$ is a multiplicative function so we have that its
Dirichlet inverse $\beta^{-1}$ must also be multiplicative. Hence,
to determine $\beta^{-1}(n)$ for all $n$, we can compute it for
primes $p$. 

From \cite{Apostol1976}, for $p=2$, we have 
\begin{equation}
(\beta\ast\beta^{-1})(2)=\beta(1)\beta^{-1}(2)+\beta(2)\beta^{-1}(1)=I(2)=0,
\end{equation}
where 
\begin{equation}
I(n)=\begin{cases}
1 & \text{if }n=1\\
0 & \text{if }n>1
\end{cases}
\end{equation}
the identity function, giving us $\beta^{-1}(2)=0$. For $p\neq2$,
we have 
\begin{align*}
(\beta\ast\beta^{-1})(p)=\beta(1)\beta^{-1}(p)+\beta(p)\beta^{-1}(1)=I(p)=0,
\end{align*}
giving us $\beta^{-1}(p)=-\beta^{-1}(1)=-1$. It follows from the
multiplicative property of $\beta^{-1}$ that 
\begin{equation}
\beta^{-1}(n)=\begin{cases}
0 & \text{if }n\text{ even}\\
\mu(n) & \text{if }n\text{ odd}
\end{cases},
\end{equation}
where 
\begin{equation}
\mu(n)=\begin{cases}
1 & \text{if }n=1\\
(-1)^{k} & \text{if }n\text{ is the product of }k\text{ distinct primes}\\
0 & \text{if }n\text{ is divisible by a square }>1
\end{cases},
\end{equation}
is the Möbius function. 

We now make use of the identity 
\begin{equation}
\sum_{d|n}\alpha(d)\alpha^{-1}\left(\frac{n}{d}\right)=I(n),
\end{equation}
where $\alpha$ is an arithmetical function with $\alpha(1)\neq0$
with Dirichlet inverse $\alpha^{-1}$ \cite{Apostol1976}: 
\begin{align*}
R(\lambda) & =\sum^{\infty}_{k=1}I(k)R\left(\frac{\lambda}{k}\right)\\
 & =\sum^{\infty}_{k=1}\left(\sum_{d|k}\beta^{-1}(d)\beta\left(\frac{k}{d}\right)\right)R\left(\frac{\lambda}{k}\right),
\end{align*}
and substituting $k=nm$ gives us 
\begin{equation}
\sum^{\infty}_{m=1}\sum^{\infty}_{n=1}\beta^{-1}(n)\beta(m)R\left(\frac{\lambda}{nm}\right).\label{eq:beta sum}
\end{equation}
We show that these sums can in fact be reordered. From $|R(\lambda)|\leq C\lambda^{2}$
on $[0,1]$, we have by the smoothness of $R$ that $\lambda\in[0,k]$
implies $|R(\lambda)|\leq C_{k}\lambda^{2}$, with $C_{k}>0$, for
all $k\in\mathbb{N}$. Hence, on any interval $[0,k]$, $k\in\mathbb{N}$,
we have 
\begin{equation}
|\beta^{-1}(n)\beta(m)R\left(\frac{\lambda}{nm}\right)|\leq|R\left(\frac{\lambda}{nm}\right)|\leq\frac{C_{k}\lambda^{2}}{(nm)^{2}},
\end{equation}
and 
\begin{equation}
\sum^{\infty}_{m=1}\sum^{\infty}_{n=1}\frac{C_{k}\lambda^{2}}{(nm)^{2}}=\frac{C_{k}\lambda^{2}\pi^{4}}{36}.
\end{equation}
It follows by the Weierstrass $M$-Test that the sum \prettyref{eq:beta sum}
is absolutely convergent on any interval $[0,k]$, $k\in\mathbb{N}$.
Hence, on $[0,\infty)$ we have 
\begin{align*}
R(\lambda) & =\sum^{\infty}_{m=1}\sum^{\infty}_{n=1}\beta^{-1}(n)\beta(m)R\left(\frac{\lambda}{nm}\right)\\
 & =\sum^{\infty}_{n=1}\sum^{\infty}_{m=1}\beta^{-1}(n)\beta(m)R\left(\frac{\lambda}{nm}\right)\\
 & =\sum^{\infty}_{n=1}\beta^{-1}(n)\left(\sum^{\infty}_{m=1}\beta(m)R\left(\frac{\lambda}{nm}\right)\right)\\
 & =\sum^{\infty}_{n=1}\beta^{-1}(n)\tilde{R}\left(\frac{\lambda}{n}\right),
\end{align*}
showing that $\tilde{R}$ uniquely determines $R$ on $[0,\infty)$.
\end{proof}

\section{ Spectral Determination of a Radially Symmetric Magnetic Field via
the heat kernel}\label{sec:Spectral-Determination-of via heat kernel}

The proof of \prettyref{thm: Theorem 3} in the previous section used
a rescaled trace expansion \prettyref{eq:rescaled heat kernel exp.}
of the Schrödinger operator $\hat{S}_{h}$ corresponding to an arbitrary
compactly supported function $f\in C^{\infty}_{c}\left(\mathbb{R}\right)$. 

One may ask if the same is provable using the heat trace expansion.
However, the heat operator of the magnetic Laplacian $e^{-t\hat{S}_{h}}$
is not always trace class; even for $B$ a potential well. This requires
additional assumptions on the magnetic field $B$. The hypotheses
are:

\begin{align}
\textnormal{\textbf{H1:}}\qquad & \exists C>0\quad\textnormal{such that }\left|B\left(x\right)\right|\leq C\left|B\left(x'\right)\right|,\,\forall\left\Vert x-x'\right\Vert \leq1\label{eq:Hyp 1}\\
\textnormal{\textbf{H2:}}\qquad & \max_{\left|\alpha\right|=2}\left(\sup_{\left|x-x'\right|\leq1}\left\Vert \partial^{\alpha}a\right\Vert \right)=O\left(\left|B\left(x\right)\right|^{3/2}\right)\label{eq:Hyp 2}\\
 & \qquad\quad\textnormal{ for a suitable chosen magnetic potential \ensuremath{a}}\nonumber \\
\textnormal{\textbf{H3:}}\qquad & \int_{\mathbb{R}^{2}}e^{-tB\left(x\right)}dx<\infty,\,\forall t>0,\quad\textnormal{with }B=\left|B\left(x\right)\right|.\label{eq:Hyp 3}
\end{align}
It follows from the results of \cite{cdv1986} that the heat operator
of the magnetic Laplacian $e^{-t\hat{S}_{h}}$, $\forall t>0$, is
trace class under the above hypothesis H1, H2 and H3 for $B$ a potential
well.

For the sake of our inverse spectral results we shall assume that
$B=B\left(r\right)$ is radial. Furthermore, we introduce a slight
strengthening of the last hypothesis H3 as follows
\begin{equation}
\textnormal{\textbf{H3':}}\qquad\int^{\infty}_{1}r\,B(r)^{1-\sigma}\,dr<\infty,\quad\textnormal{for some }\sigma>1.\label{eq:Hyp 3'}
\end{equation}
It is clear that H3' implies H3.

Our theorem is as follows.
\begin{thm}
\label{thm:Theorem 4} Suppose $B$ is a radially symmetric magnetic
potential well in 2 dimensions, with $B'(r)>0$, $B(0)>0$ . Assuming
H1 \prettyref{eq:Hyp 1}, H2 \prettyref{eq:Hyp 2} and H3' \prettyref{eq:Hyp 3'},
the heat operator of the magnetic Laplacian $e^{-\frac{t}{h}\hat{S}_{h}}$
is trace class and its trace determines the magnetic field $B\left(r\right)$.
In particular, the magnetic field is determined by the spectrum of
the Schrödinger operator $\hat{S}_{h}$.
\end{thm}

The rest of the section outlines the proof of the above theorem.

Since $B\in C^{1}$, $B'(r)>0$, and $B(r)\to\infty$ as $r\to\infty$,
we have that $B$ is a bijection from $[0,\infty)$ to $[B(0),\infty)$,
and so has a $C^{1}$ inverse $f$ with $f(s)=r$, $f'(s)>0$, and
$B(f(s))=s$. Since also $f(s)>0$ and $s\geq B(0)>0$, we have 
\begin{equation}
g(s):=f'(s)\,f(s)\,s>0,
\end{equation}
for all $s>B(0)$. Substituting $s=B(r)$, and so $r=f(s)$ and $dr=f'(s)\,ds$,
we get 
\begin{align*}
\hat{B}(t):=c_{0} & =\int^{\infty}_{B(0)}\frac{s\,f(s)}{e^{ts}-e^{-ts}}\,f'(s)\,ds\\
 & =\int^{\infty}_{B(0)}\frac{g(s)}{e^{ts}-e^{-ts}}\,ds.
\end{align*}

Recall that for $s\geq B(0)>0$ and $t>0$ we have $e^{-2ts}<1$,
so 
\begin{equation}
\frac{1}{e^{ts}-e^{-ts}}=\frac{e^{-ts}}{1-e^{-2ts}}=\sum^{\infty}_{k=0}e^{-(2k+1)ts}.
\end{equation}
Since every term $g(s)\,e^{-(2k+1)ts}$ is non-negative, Tonelli's
theorem justifies the exchange of the sum and integral: 
\begin{equation}
\hat{B}(t)=\sum^{\infty}_{k=0}\int^{\infty}_{B(0)}g(s)\,e^{-(2k+1)ts}\,ds.
\end{equation}
Define the Laplace transform 
\begin{equation}
F(u):=\int^{\infty}_{B(0)}g(s)\,e^{-us}\,ds,
\end{equation}
so that 
\begin{equation}
\hat{B}(t)=\sum^{\infty}_{k=0}F\!\left((2k+1)t\right).
\end{equation}
Note that 
\begin{equation}
\int^{\infty}_{B(0)}g(s)\,s^{-\sigma}\,ds<\infty\label{eq: g condition}
\end{equation}
is exactly the assumption on $B$ after the substitution $s=B(r)$:
\begin{align*}
\int^{\infty}_{B(0)}g(s)\,s^{-\sigma}\,ds & =\int^{\infty}_{0}f'(B(r))rB(r)B(r)^{-\sigma}B'(r)\,dr\\
 & =\int^{\infty}_{0}r\,B(r)^{1-\sigma}\,dr.
\end{align*}
Since $s^{\sigma}e^{-us}$ is bounded on {[}B(0), $\infty$) for every
$u>0$, it follows that F(u) $<\infty$ and is well defined.

Let 
\begin{equation}
M[h](z)=\int^{\infty}_{0}h(t)\,t^{z-1}\,dt
\end{equation}
denote the Mellin transform of $h$ when it exists. We compute 
\begin{equation}
M[\hat{B}](z)=\int^{\infty}_{0}\sum^{\infty}_{k=0}t^{z-1}F((2k+1)t)\,dt,
\end{equation}
(which we will also show to be well-defined), fix $\operatorname{Re}(z)=\sigma$,
and claim the sum and integral may be exchanged. Namely that 
\begin{equation}
M[\hat{B}](z)=\sum^{\infty}_{k=0}\int^{\infty}_{0}t^{z-1}F\!\left((2k+1)t\right)dt.
\end{equation}
We have $|t^{z-1}|=t^{\sigma-1}$ and $F(u)\geq0$ since $g>0$, so
by Tonelli's theorem it suffices to show that 
\begin{equation}
|M[\hat{B}](z)|\leq\int^{\infty}_{0}\sum^{\infty}_{k=0}t^{\sigma-1}F((2k+1)t)\,dt<\infty,
\end{equation}
since absolute convergence will then allow us to apply Fubini's theorem.
Substituting the definition of $F$, we have 
\begin{equation}
\int^{\infty}_{0}\sum^{\infty}_{k=0}t^{\sigma-1}F((2k+1)t)\,dt=\int^{\infty}_{0}\sum^{\infty}_{k=0}t^{\sigma-1}\int^{\infty}_{B(0)}g(s)\,e^{-(2k+1)ts}\,ds\,dt,
\end{equation}
and Tonelli's theorem again allows us to exchange the sum and integrals,
giving 
\begin{equation}
\sum^{\infty}_{k=0}\int^{\infty}_{B(0)}g(s)\int^{\infty}_{0}t^{\sigma-1}e^{-(2k+1)ts}\,dt\,ds.
\end{equation}
Substituting $w=(2k+1)st$, we get 
\begin{align*}
\int^{\infty}_{0}t^{\sigma-1}e^{-(2k+1)st}\,dt & =\big((2k+1)s\big)^{-\sigma}\int^{\infty}_{0}w^{\sigma-1}e^{-w}\,dw\\
 & =\Gamma(\sigma)\,\big((2k+1)s\big)^{-\sigma},
\end{align*}
and hence 
\begin{align*}
\sum^{\infty}_{k=0}\int^{\infty}_{B(0)}g(s)\int^{\infty}_{0}t^{\sigma-1}e^{-(2k+1)ts}\,dt\,ds & =\sum^{\infty}_{k=0}\int^{\infty}_{B(0)}g(s)\,\Gamma(\sigma)\,\big((2k+1)s\big)^{-\sigma}\,ds\\
 & =\Gamma(\sigma)\left(\sum^{\infty}_{k=0}(2k+1)^{-\sigma}\right)\int^{\infty}_{B(0)}g(s)\,s^{-\sigma}\,ds.
\end{align*}
We have that $\Gamma(\sigma)$ is finite since $\sigma>1>0$, that
$\sum^{\infty}_{k=0}(2k+1)^{-\sigma}$ is finite since $\sigma>1$,
and that $\int^{\infty}_{B(0)}g(s)s^{-\sigma}\,ds$ is finite by \prettyref{eq: g condition}.
Hence 
\begin{equation}
\int^{\infty}_{0}\sum^{\infty}_{k=0}t^{\sigma-1}F((2k+1)t)\,dt=\Gamma(\sigma)\left(\sum^{\infty}_{k=0}(2k+1)^{-\sigma}\right)\int^{\infty}_{B(0)}g(s)\,s^{-\sigma}\,ds<\infty,
\end{equation}
justifying the application of Fubini's theorem to exchange the sum
and integral. In particular, this also shows that $M[\hat{B}](z)$
converges absolutely on $\operatorname{Re}(z)=\sigma$, verifying
that the Mellin transform of $\hat{B}$ is well-defined.

Substituting $u=(2k+1)t$ in each term of the sum now gives 
\begin{equation}
\int^{\infty}_{0}t^{z-1}F((2k+1)t)\,dt=(2k+1)^{-z}\int^{\infty}_{0}u^{z-1}F(u)\,du,
\end{equation}
so 
\begin{equation}
M[\hat{B}](z)=A(z)\,M[F](z),\qquad A(z):=\sum^{\infty}_{k=0}(2k+1)^{-z}.
\end{equation}
We have that the series $A(z)$ converges absolutely for $\operatorname{Re}(z)>1$
since $|(2k+1)^{-z}|=(2k+1)^{-\sigma}$. Note that $A(z)=(1-2^{-z})\zeta(z)$,
the Dirichlet lambda function, is nonzero for $\operatorname{Re}(z)>1$:
$\zeta(z)=\prod_{p}(1-p^{-z})^{-1}$ is nonzero and converges absolutely
for $\operatorname{Re}(z)>1$, and $|2^{-z}|=2^{-\operatorname{Re}(z)}<1$
for $\operatorname{Re}(z)>1$ implies $1-2^{-z}\neq0$ for $\operatorname{Re}(z)>1$.\\
 \hspace*{0.5cm}Since $M[\hat{B}](z)$ is determined by $\hat{B}$,
and $A(z)$ is known and nonzero, we can conclude that $M[F](z)$
is determined by $\hat{B}$. Moreover, since $M[\hat{B}](z)=A(z)M[F](z)$
converges absolutely on $\operatorname{Re(z)=\sigma}$ and $A(z)\neq0$,
it follows that $M[F](z)$ is indeed well-defined on this line.

Furthermore, 
\begin{equation}
\int^{\infty}_{0}\int^{\infty}_{B(0)}|u^{z-1}g(s)e^{-us}|ds\ du=\int^{\infty}_{0}\int^{\infty}_{B(0)}u^{\sigma-1}g(s)e^{-us}ds\ du,
\end{equation}
and since $u^{\sigma-1}g(s)e^{-us}\geq0$, Tonelli's theorem gives
us 
\begin{equation}
\int^{\infty}_{B(0)}g(s)\left(\int^{\infty}_{0}u^{\sigma-1}e^{-us}du\right)ds=\Gamma(\sigma)\int^{\infty}_{B(0)}g(s)s^{-\sigma}ds<\infty
\end{equation}
by substituting $v=us$. Hence the double integral converges absolutely
and by Fubini's theorem we have 
\begin{align*}
M[F](z) & =\int^{\infty}_{0}u^{z-1}\int^{\infty}_{B(0)}g(s)\,e^{-us}\,ds\,du\\
 & =\int^{\infty}_{B(0)}g(s)\int^{\infty}_{0}u^{z-1}e^{-us}\,du\,ds.
\end{align*}
Substituting $v=us$ in the inner integral gives 
\begin{align*}
\int^{\infty}_{0}u^{z-1}e^{-us}\,du & =s^{-z}\int^{\infty}_{0}v^{z-1}e^{-v}\,dv\\
 & =s^{-z}\,\Gamma(z).
\end{align*}
Therefore 
\begin{equation}
M[F](z)=\Gamma(z)\int^{\infty}_{B(0)}g(s)\,s^{-z}\,ds.
\end{equation}
Since $\Gamma(z)$ is nonzero for all $\operatorname{Re}(z)>0$, and
the left-hand side is determined by $\hat{B}$, we conclude that 
\begin{equation}
\tilde{M_{g}}(z):=\int^{\infty}_{B(0)}g(s)\,s^{-z}\,ds
\end{equation}
is determined by $\hat{B}$.

We now show that $g$ is uniquely determined by $\tilde{M_{g}}$.
It suffices to determine $\tilde{M_{g}}$ on the vertical line $\operatorname{Re}(z)=\sigma$,
as knowing $\tilde{M_{g}}(\sigma+i\xi)$ for all $\xi\in\mathbb{R}$
is equivalent to knowing the Fourier transform of a certain $L^{1}$
function, as we now show, and Fourier injectivity then recovers $g$
uniquely.

Write $z=\sigma+i\xi,\ \xi\in\mathbb{R}$. Substitute $s=e^{x}$,
and so $ds=e^{x}\,dx$, and let $H(x):=g(e^{x})\,e^{(1-\sigma)x}$.
Then 
\begin{align*}
\tilde{M_{g}}(\sigma+i\xi) & =\int^{\infty}_{\log B(0)}g(e^{x})\,e^{-\sigma x}\,e^{-i\xi x}\,e^{x}\,dx\\
 & =\int^{\infty}_{\log B(0)}H(x)\,e^{-i\xi x}\,dx.
\end{align*}
Extending $H$ by zero on $(-\infty,\log B(0))$, we see that this
is the Fourier transform of an $L^{1}(\mathbb{R})$ function, since
\begin{equation}
\int^{\infty}_{-\infty}|H(x)|\,dx=\int^{\infty}_{\log B(0)}g(e^{x})\,e^{(1-\sigma)x}\,dx=\int^{\infty}_{B(0)}g(s)\,s^{-\sigma}\,ds<\infty,
\end{equation}
where the last equality uses the substitution $s=e^{x}$, and the
growth assumption on B. By the uniqueness theorem for Fourier transforms
on $L^{1}(\mathbb{R})$, two $L^{1}$ functions with the same Fourier
transform agree almost everywhere; since $H$ is continuous (follows
from $f$ and $e^{(1-\sigma)x}$ being $C^{1}$, and the limits at
$\log B(0)$ from the left and right agreeing), this a.e. equality
is in fact equality everywhere. Hence the Fourier transform of $H$
uniquely determines $H$ everywhere, and hence $g$ everywhere on
$[B(0),\infty)$.

We have recovered $g(s)=f'(s)\,f(s)\,s$. Observing that 
\begin{equation}
g(s)=\frac{s}{2}\,\frac{d}{ds}\!\left[f(s)^{2}\right]
\end{equation}
and $f(B(0))=0$, we integrate to obtain 
\begin{equation}
f(s)^{2}=2\int^{s}_{B(0)}\frac{g(u)}{u}\,du,
\end{equation}
which determines $f=B^{-1}$ and hence $B=f^{-1}$ uniquely.

The result and general proof structure can be adapted to the 3-dimensional
case where instead one wishes to show that $B$ is uniquely determined
by 
\begin{equation}
\tilde{B}(t)=\frac{1}{\sqrt{t}}\int^{\infty}_{0}\frac{B(r)\,r^{2}}{e^{tB(r)}-e^{-tB(r)}}\,dr.
\end{equation}

The $t^{-1/2}$ factor just shifts the Mellin transform argument by
$-\tfrac{1}{2}$, and $g$ has an extra power of $f$. Additionally
the assumption on $B$ changes slightly to 
\begin{equation}
\int^{\infty}_{0}r^{2}\,B(r)^{1-\sigma}\,dr<\infty,
\end{equation}
for some $\sigma>\frac{3}{2}$, a growth condition on $B$ of a similar
kind as before.

Again one makes the substitution $s=B(r)$: 
\begin{equation}
\widetilde{B}(t)=\frac{1}{\sqrt{t}}\int^{\infty}_{B(0)}\frac{f(s)^{2}\,s}{e^{ts}-e^{-ts}}\,f'(s)\,ds=\frac{1}{\sqrt{t}}\int^{\infty}_{B(0)}\frac{g(s)}{e^{ts}-e^{-ts}}\,ds,
\end{equation}
where now $g(s):=f'(s)\,f(s)^{2}\,s$ (note the extra power of $f$).
The new growth condition 
\begin{equation}
\int^{\infty}_{0}r^{2}\,B(r)^{1-\sigma}\,dr<\infty
\end{equation}
for some $\sigma>\tfrac{3}{2}$ is equivalent to 
\begin{equation}
\int^{\infty}_{B(0)}g(s)\,s^{-\sigma}\,ds<\infty.
\end{equation}
The geometric series step is unchanged, giving 
\begin{equation}
\widetilde{B}(t)=t^{-1/2}\sum^{\infty}_{k=0}F\!\left((2k+1)t\right),
\end{equation}
and the extra factor of $t^{-1/2}$ shifts the Mellin transform by
$-\tfrac{1}{2}$: 
\begin{equation}
M[\widetilde{B}](z)=\int^{\infty}_{0}t^{(z-\frac{1}{2})-1}\cdot\sum^{\infty}_{k=0}F\!\left((2k+1)t\right)dt,
\end{equation}
The rest of the argument follows, with $z$ replaced by $z-\tfrac{1}{2}$
throughout, and one needs $\operatorname{Re}(z)-\tfrac{1}{2}>1$,
i.e.\ $\operatorname{Re}(z)>\tfrac{3}{2}$, for convergence. The
same conclusion is reached, namely that $\int^{\infty}_{B(0)}g(s)\,s^{-z}\,ds$
is determined by $\widetilde{B}$, and Fourier inversion gives $g$.

To recover $B$ from $g$, now $g(s)=f'(s)\,f(s)^{2}\,s$, we note
that 
\begin{equation}
g(s)=\frac{s}{3}\,\frac{d}{ds}\!\left[f(s)^{3}\right],
\end{equation}
and integrating gives us 
\begin{equation}
f(s)^{3}=3\int^{s}_{B(0)}\frac{g(u)}{u}\,du.
\end{equation}
This uniquely recovers $f$ and hence $B=f^{-1}$.

\begin{rem} The assumption $B(0)>0$ is imposed to avoid dealing
with the behaviour of the Mellin integrals at $s=0$. We expect the
method to remain valid under suitable assumptions allowing $B(0)=0$,
for example $B(r)\sim c\,r^{m}$ as $r\to0$ for some $c>0$ and $m>0$.
Establishing this would require a separate analysis of the small-$s$
behaviour of the Mellin transforms. \end{rem}

\bibliographystyle{siam}
\bibliography{biblio}

\end{document}